\documentclass[12pt]{amsart}

\usepackage[colorlinks]{hyperref}
\hypersetup{allcolors=blue}
\usepackage{enumerate}
\usepackage{geometry}
\usepackage{color} 
\usepackage{tikz}
\usepackage[all,knot,arc,import,poly]{xy}

\usepackage{comment}

\usepackage{pgfplots}
\pgfplotsset{compat=1.15}
\usetikzlibrary{arrows}

\usepackage{amsrefs}

\def\calE{\mathcal{E}}
\def\calO{\mathcal{O}}
\def\bbA{\mathbb{A}}
\def\bbF{\mathbb{F}}
\def\bbG{\mathbb{G}}
\def\bbP{\mathbb{P}}
\def\bbR{\mathbb{R}}
\def\bbZ{\mathbb{Z}}
\def\Pic{\text{Pic}}
\def\Proj{\text{Proj}}
\def\Spec{\text{Spec}}
\def\Supp{\text{Supp}}

\everymath{\displaystyle}

\newtheorem{thm}{Theorem}[section]
\newtheorem{prop}[thm]{Proposition}

\newtheorem{rmk}[thm]{Remark}
\newtheorem{lm}[thm]{Lemma}

\title{The moduli space of rational elliptic surfaces of index two}
\author{Rick Miranda}
\address{Department of Mathematics, Colorado State University}
\email{rick.miranda@colostate.edu}
\author{Aline Zanardini}
\address{Mathematical Institute, Leiden University}
\email{a.zanardini@math.leidenuniv.nl}
\date{\today}

\begin{document}

\begin{abstract}
In this paper we construct a moduli space for marked rational elliptic surfaces of index two as a non-complete toric variety of dimension nine. We also construct compactifications of this moduli space, which are obtained as quotients of $\bbA^{12}$ by an action of $\bbG_m^3$.
\end{abstract}

\maketitle

\tableofcontents

\section{Introduction}
We say a smooth and projective rational surface $Y$ is a rational elliptic surface (RES)
if $Y$ admits a relatively minimal fibration $\calE:Y \to \bbP^1$
whose generic fiber is a smooth curve of genus one.
We do not necessarily assume the existence of a global section.
If $Y$ is a rational elliptic surface,
then there exists some $m\geq 1$, called the index of the fibration,
so that $\mathcal{E}$ is given by the anti-pluricanonical system $|-mK_Y|$.
Moreover, $m=1$ if and only if $\mathcal{E}$ admits a global section
and whenever $m>1$ there exists exactly one multiple fiber in $Y$,
which is of multiplicity $m$ (see e.g. \cite[Chapter V, \S 6]{dc}).
In this paper we are interested in the case $m=2$.

Rational elliptic surfaces of index $m$ can be realized
as a nine-fold blow-up of $\mathbb{P}^2$,
where the nine points are base points of a Halphen pencil (of index $m$)
\cite[Theorem 5.6.1]{dc} -- these are pencils of plane curves of degree $3m$
having nine (possibly infinitely near) base points of multiplicity $m$.
The multiple fiber corresponds to the (unique) cubic curve through the nine points.
Any RES of index $m$ comes with a rational $m$-section $\theta$,
which can be taken to be the exceptional curve of the last of the nine blowups.
It is not uniquely determined by $\mathcal{E}$.
If $Y$ is a RES of index $m$ with a given rational $m$-section $\theta$,
we will call the pair $(Y,\theta)$ a \emph{marked RES} of index $m$.

Counting parameters, we see that marked RESs of index one (also called Weierstrass fibrations)  depend on $8$ parameters -- these are parametrized by an open subset of the Grassmannian $G(2,10)$ of pencils of plane cubics. However, the construction of the Jacobian surface shows marked RESs of index at least two depend on $9$ parameters. Essentially, we have the additional choice of the multiple fiber.

In \cite{mirandaW} the moduli space of marked RESs of index one was constructed via geometric invariant theory (GIT). A different compactification was later obtained in \cite{hl}, where it is shown that the moduli space is rational. More recently, modular compactifications via stable slc pairs have also been constructed in \cite{alexeevK3}, \cite{alexeevADE} and \cite{ascher}. In this article we address the problem of constructing a moduli space for marked RESs of index two. 

When the choice of a marked bisection is not part of the classification problem, then a possible approach to constructing the corresponding moduli space has been considered in \cite{az}.  The addition of the additional data of the chosen bisection - the marking - seems to rigidify things so that all the RES's of index two are semi-stable.  This is not the case for the unmarked analysis (\cite{az}) nor was the case for the RES's of index one (the Weierstrass fibrations considered in \cite{mirandaW}).

First, in Section \ref{double}, we describe each marked RES of index two
as a double cover of $\bbP^1\times\bbP^1$,
and reduce the moduli problem to that of the branch curve.
Next in Section \ref{sec:normalforms} 
we normalize the branch curve and prove that the moduli space
corresponds to a quotient of $\bbA^{10}$ by $\bbG_m \times \bbZ/2\bbZ$,
which is realized as an explicit toric variety in Section \ref{z2}.
This leads to a non-compact quotient space though.

We then relax the normalization of the equation for the branch curve
and prove in Section \ref{sec:A12quotient}
that the moduli space also corresponds to a specific quotient of $\bbA^{12}$
by $\bbG_m^3$.  This leads to a multitude of toric GIT quotients,
and we identify those that compactify the quotient space obtained with the first approach.
The description is completely explicit, defined by a fan with prescribed cones.
In Section \ref{unstable} we also identify the unstable loci 
and show that all branch curves that give rise to RESs of index two 
are stable points for the action.  Our main result may be presented as follows.

\begin{thm}\label{thm:main}
A moduli space for marked RES's of index two
exists as a non-complete toric variety of dimension $9$,
obtained as a quotient of a single affine toric variety by an action of $\bbG_m$.
Compactifications of this moduli space exist as complete toric varieties,
obtained as quotients of $\bbA^{12}$ by an action of $\bbG_m^3$.
\end{thm}

\section{RES of index two as double covers of $\bbP^1\times \bbP^1$}
\label{double}

Consider $\bbP^1\times\bbP^1$,
and let $h$ (resp. $f$) denote the class of a section (resp. of a fiber).
Let $\bbF$ be the non-minimal rational surface obtained from $\bbP^1\times \bbP^1$
by blowing-up a point $p$ and then blowing-up an infinitely near point (of order 1)
at the intersection of the exceptional curve
and the proper transform $\bar{f}$ of the fiber through $p$. 

Let $e_2$ (resp. $e_1$) denote the class of the second exceptional divisor
(resp. the proper transform of the first);
note that $\bar{f}\sim f-2e_2-e_1$.
Choose a reduced member $C\in |4h+3f-4e_2-2e_1| = |4h+3\bar{f}+2e_2 +e_1|$
with only ade singularities,
that does not have either $\bar{f}$ or $e_1$ as a component. We will show (Theorem \ref{isomclassesbranchcurve}) there is a $1$-$1$ correspondence between
isomorphism classes of marked RESs of index two
and isomorphism classes of such curves $C$.

The condition that $C$ does not contain $e_1$ as a component
means that the singularity of $C$ at $p$ is a double point. In fact $C$ corresponds to a curve on $\bbP^1\times\bbP^1$ of bidegree $(4,3)$
with a tacnode tangent to the fiber through $p$. In particular, when $C$ is irreducible, then by the genus formula, asking for $C$ to have only ade singularities is redundant.

Starting from $\bbF$ and $C$ as above, let $\pi:Y' \to \bbF$ be the double cover map
with branch locus $\Delta$ equal to $C+\bar{f}+e_1 \sim 2(2h + 2\bar{f}+e_2+e_1)\sim 2(2h+2f-3e_2-e_1)$. 
Since $(C \cdot \bar{f}) = (C \cdot e_1)=0$
the branch locus is smooth along $\bar{f}$ and $e_1$.

The pencil $|f|$ lifts to an elliptic fibration on $Y'$, since $(f \cdot \Delta) = 4$.
The elliptic fiber lying over the fiber through $p$ is $\pi^*(\bar{f}+e_1+2e_2)$.

Now, because $K_{\bbF}=-2h-2f+2e_2+e_1$, it follows that
\[
K_{Y'}=\pi^{*}K_{\bbF}+\Delta'=\pi^{*}(K_{\bbF}+\Delta/2)=\pi^*(-e_2)
\]
where $\Delta'$ is the ramification locus of $\pi$.
Thus, by adjunction, the curve $E:= \pi^*(e_2)$ is a curve of genus one.

The curves $\bar{f}$ and $e_1$ lift to $(-1)$ curves $D_1$ and $D_2$ respectively in $Y'$;
let $\mu: Y'\to Y$ denote their contraction.
Then $\mu^*(\mu_*K_{Y'})=K_{Y'}-D_1-D_2$ and $K_Y=\mu_*K_{Y'}$,
and using the projection formula we obtain
\[
K_Y^2=(\mu_*K_{Y'})^2=(K_{Y'}-D_1-D_2)\cdot K_{Y'}=\pi^*(-e_2)\cdot \pi^*(-e_2)+2=2e_2^2+2=0
\]
which implies that $Y$ is a (relatively minimal) rational elliptic surface of index two
with multiple fiber $2\mu_*E$.
Note that the genus of $E=\pi^*(e_2)$ does not change under $\mu$. 

Consider finally the proper transform $\bar{h}$ on $\bbF$ 
of the horizontal section through the point $p$; that is, $\bar{h} \sim h-e_1-e_2$.
We note that $C$ cannot have $\bar{h}$ as a component,
otherwise the residual curve $C-\bar{h}$ would satisfy $(C-\bar{h})\cdot e_1=-1$,
so that it would have to contain $e_1$ as a component as well, which we are forbidding.
Hence $\bar{h}$ is not part of the branch locus of the double cover.
We have $\bar{h}^2=-1$ and $(\bar{h}\cdot C) = 1$,
so that $(\bar{h}\cdot \Delta) = 2$.
These two intersection points are distinct: one is the intersection point with $C$
and one is the intersection point with $e_1$.
Hence $\theta = \pi^*(\bar{h})$ is
a bisection of the elliptic fibration.
It survives in the blowdown to $Y$ as a smooth rational bisection with self-intersection $-1$.

Conversely, we can prove

\begin{prop}
Every rational elliptic surface of index two  $Y\to \mathbb{P}^1$ arises this way.
In particular, it fits in a diagram:
\[
\xymatrix{
Y & \ar[l]_{\mu} Y' \ar[d]_\pi^{2:1} & \\
   & \bbF\ar[r]_{\varphi} & \bbP^1\times \bbP^1
}
\]
\label{converse}
\end{prop}

\begin{proof}
Let $Y\to \bbP^1$ be a RES of index two with multiple fiber $2F_0$, and choose a smooth rational bisection, say $\theta$. On (Pic of) the generic fiber $Y_{\eta}$ we can consider the involution $p \mapsto \theta-p$, which extends to an involution $\iota$ of $Y$ and whose fixed locus consists of a $4-$section and two isolated points lying on $F_0$.
The rational bisection passes through one of the two fixed points for $\iota$.
(In fact, by Riemann-Roch $\theta$ restricted to $Y_{\eta}$, we have a degree two map $Y_{\eta} \to \mathbb{P}^1$ and the involution we are considering is simply the corresponding covering involution.
Then, by Riemann-Hurwitz, on each smooth fiber the fixed locus consists of four points, hence there exists a $4-$section in Fix($\iota$). Note however that such $4-$section must meet $F_0$ at two points, and thus we have two other isolated fixed points lying on $F_0$.)

Denote by $\mu: Y'\to Y$ the blow-up of these two points
and by $\iota$ the induced involution on $Y'$.
The quotient surface $\bbF\doteq Y'/\iota$ is a rational surface
which comes with a $\bbP^1$-fibration to $\bbP^1$ 
and by construction the branch locus is of the form $C+\delta_1+\delta_2$,
where $C$ is a $4-$section (which is the image of the $4$-section on $Y$)
and $\delta_i$ are the images of the two exceptional curves.
We may assume that $\delta_1$ corresponds to the fixed point through which the given rational bisection passes.
In particular, $\delta_i^2=-2$.
The curve $e$ in $\bbF$ which corresponds to $F_0$
is a $(-1)$-curve so we can contract it;
both $\delta_1$ and $\delta_2$ will then become $(-1)$-curves.
We then contract $\delta_1$;
at this point the curve $\delta_2$ will be such that $\delta_2^2=0$,
and gives a fiber of the $\bbP^1$-fibration.
The curve $C$ now has a tacnode at a point $p$ on $\delta_2$,
with a vertical tangent.
The bisection descends to the horizontal section of $\bbF_0$ that passes through the tacnode point $p$.
Hence we have a birational morphism $\varphi: \bbF\to \bbF_N$ for some $N$.
We claim that up to reversing the roles of $\delta_1$ and $\delta_2$ we must have $N=0$ and the branch curve (hence $C$) is as described at the beginning of this section. 
If we choose $\{f,h,e,\delta_1\}$ as basis for $\Pic(\bbF$) 
then $C \sim 4h+3f-4e-2\delta_1$ and $C$ does not contain $\delta_1$ or $\delta_2$ or the horizontal section through $p$.
\end{proof}

\begin{rmk}
More generally, if $S \to \mathbb{P}^1$ is any elliptic surface with two multiple fibers of multiplicities $2m_1$ and $m_2$ ($m_2$ odd) and a multisection of degree $2m_1m_2$ (gcd$(2m_1,m_2)=1$), then one can construct another elliptic surface  $J^{m_1m_2}(S)\to \mathbb{P}^1$ whose generic fiber consists of the set of line bundles of degree $m_1m_2$ on the generic fiber of $S$. This surface comes with a natural involution determined by the   multisection and it has exactly one multiple fiber of multiplicity two. In \cite{friedman}, Friedman shows $J^{m_1m_2}(S)$ can be described as a double cover of $\mathbb{F}_{N}$ branched over a curve in the linear system $|4\sigma+(2k+1)f-4e_2-2e_1|$, where $\sigma$ is the negative section and the numbers $k$ and $N$ satisfy $p_g(S)=k-N-1$. In Proposition \ref{converse} we consider the particular case when $S$ is rational, $m_1=m_2=1$ and hence $S \simeq J^{m_1m_2}(S)$. 
\end{rmk}

Another way of describing how a RES of index two can be realized as double covers of $\mathbb{P}^1\times \mathbb{P}^1$ is presented below.

Fix the pair $(Y,\theta)$ as above.
Denote by $F_0$ the (reduced) fiber that is the double fiber,
so that the fiber class is $F = 2F_0$;
$F_0$ is also the anti-canonical class $-K_Y$.
Note that since $\theta$ is smooth and rational, and $\theta\cdot K_Y = -1$, we must have $\theta^2 = -1$.
Let $p = \theta \cap F_0$ be the point of intersection of $\theta$ with the double fiber;
since $\theta\cdot F = 2$, we have $\theta \cdot F_0 = 1$, so that
$\theta$ and $F_0$ meet transversally at $p$.

Since $F_0$ does not move on $Y$, but $2F_0 = F$ does move in a base-point-free pencil,
the normal bundle of $F_0$ in $Y$ is a two-torsion class $\tau$ in the Picard group of $F_0$.
We may write $\tau$ as a divisor, uniquely as $q-p$,
where $p$ is the point of intersection with $\theta$ as noted above,
and $q$ is another (distinct from $p$) point on $F_0$.
(i.e., $q$ is uniquely determined by the pair $(Y,\theta)$.)

We note that the line bundle $\calO_Y(\theta)$ has trivial $H^1$, which follows from the restriction sequence to $\theta$, noting that $\theta^2=-1$:
\[
0 \longrightarrow \calO_Y \longrightarrow \calO_Y(\theta) \longrightarrow \calO_{\theta}(-1) \longrightarrow 0;
\]
since the left and the right sheaves have trivial $H^1$, so does the middle sheaf.

Consider now the divisor class $\theta+F_0$ on $Y$.
We have the short exact sequence
\[
0 \longrightarrow \calO_Y(\theta) \longrightarrow \calO_Y(\theta+F_0) \longrightarrow \calO_{F_0}(p+\tau) \longrightarrow 0
\]
which then shows that $H^1( \calO_Y(\theta+F_0)) = 0$ too, since the left and right sheaves have trivial $H^1$.
Since the left and right sheaves have one-dimensional $H^0$,
we conclude that $h^0( \calO_Y(\theta+F_0)) = 2$, so that $|\theta+F_0|$ is a pencil on $Y$.

This pencil has no fixed components.
Indeed, the only possible fixed components are the curves $\theta$ and $F_0$,
and if either one is fixed, the residual (which is the other one) would move in a pencil,
and neither curve does.
Since $\theta^2=-1$, $\theta \cdot F_0=1$, and $F_0^2=0$, we see that $(\theta+F_0)^2=1$,
so there is exactly one simple base point to the pencil $|\theta+F_0|$.

That base point must be the point $q$; the restriction sequence to $F_0$ above shows that
$\theta+F_0$ restricts to $F_0$ in the sheaf $\calO_{F_0}(p+\tau) = \calO_{F_0}(q)$,
which has a unique section (that vanishes at the point $q$).
Hence every section of $\calO_Y(\theta+F_0)$ vanishes at $q$.

The general member of $|\theta+F_0|$ is a smooth genus one curve, passing through $q$.
(This follows from adjunction since $K_Y = -F_0$
so that $(\theta+F_0)\cdot(\theta+F_0+K_Y) = (\theta+F_0)\cdot \theta  = 0$.)

If we blow $Y$ up at the point $q$, we obtain a rational surface $\tilde{Y}$,
and the pencil $|\theta+F_0|$ lifts to a base-point free pencil $|E|$
whose general member is curve of genus one; we have $E^2 = 0$ now.

We now have two maps from $\tilde{Y}$ to $\mathbb{P}^1$:
one comes from the original elliptic fibration (given by the linear system $|F|$);
and the other comes from the pencil $|E|$.
This gives a regular map $\pi:\tilde{Y} \to \mathbb{P}^1 \times \mathbb{P}^1$.

The general fiber of $\pi$ is obtained by intersecting two general elements of the two pencils; and
since $E \cdot F = (\theta+F_0)\cdot F = 2$, we conclude $\pi$ is a double cover.

Hence we have proved the following.

\begin{thm}\label{isomclassesbranchcurve}
There is a $1$-$1$ correspondence between
isomorphism classes of marked RESs of index two
and
isomorphism classes of curves $C$ in $\bbP^1\times\bbP^1$ of bidegree $(4,3)$
with a double point tacnode with vertical tangent,
such that $C$ does not contain the fiber or the horizontal section through the tacnode,
and has only ade singularities.
\end{thm}

\section{Normal forms for the branch locus}
\label{sec:normalforms}

The curve $C$ is (the partial resolution of) a $(4,3)$ curve
with a double point singularity which is at least a tacnode at $p$ 
(i.e. locally analytically we have $w^2+z^4=0$) 
and which is tangent to the fiber $f$ at $p$ 
(so that when resolving the singularity 
we blow-up a point in the intersection of the fiber and the first exceptional curve).

Therefore, we may choose coordinates $([x:y],[u:v])$ in $\bbP^1\times \bbP^1$
which places the point $p$ at $x=u=0$
and the fiber $f$ through $p$ is defined by $u=0$.
The curve $C$ is then given by a bihomogenous equation 
$f_C:=\sum a_{ij}u^{i}x^{j}v^{3-i}y^{4-j}=0$
of degree four in $[x:y]$ and of degree three in $[u:v]$,
and $C$ has the required tacnode at $p=([0:1],[0:1])$ with vertical tangent ($u=0$)
if and only if 
\[
a_{00}=a_{01}=a_{02}=a_{03}=a_{10}=a_{11}=0.
\]
Moreover, we may choose the coordinates so that the curve $C$ meets the horizontal section through $p$
(now defined by $x=0$) at the point $x=v=0$;
hence we may also assume that $a_{30}=0$. 

Furthermore, 

\begin{lm}
The curve $C$ must contain the monomials $u^2vy^4$ and $x^4v^3$ with nonzero coefficients. 
\label{normal}
\end{lm}

\begin{proof}
We have noted above that $C$ cannot contain
either the vertical fiber or the horizontal section through $p$ as components.  
If $a_{20}=0$, then all $a_{i0}=0$ in $f_C$, so that $x$ divides $f_C$
and the horizontal section would be a component.
If $a_{04}=0$, then all $a_{0j}=0$ in $f_C$, so that $u$ divides $f_C$
and the vertical fiber would be a component.
\end{proof}

There is another normalization of $f_C$ to exploit.
Up to changing the coordinates $[x:y]$, but fixing $[0:1]$,
we can further assume $a_{21}=0$.
In fact, a transformation of the form $[x:y] \mapsto [\alpha x:y]$
where $\alpha=-{a_{21}}/{4a_{20}}$ 
makes $a_{21}$ vanish.

All these reductions tell us we can organize the coefficients $a_{ij}$ in a matrix of the form:
\[
\begin{matrix}
 & v^3 & uv^2 & u^2v & u^3 \\
x^4 & a_{04} & a_{14} & a_{24} & a_{34} \\
x^3y &  0 & a_{13} & a_{23} & a_{33} \\
x^2y^2 & 0 & a_{12} & a_{22} & a_{32} \\
xy^3 & 0 & 0 & 0 & a_{31} \\
y^4 & 0 & 0 & a_{20} & 0 
\end{matrix}
\]

We define $Q = \{20,04,12,13,14,22,23,24,31,32,33,34\}$ to be the set of (double) indices appearing in the exponents of the terms for possible $f_C$, as above.

At this point the set of automorphisms of $\bbP^1\times\bbP^1$
which preserve this normal form for $f_C$
are of the form $([x:y],[u:v]) \mapsto ([tx:y],[su:v])$
for nonzero $s,t \in \bbG_m$.
Of course in the vector space containing such polynomials $f_C$
we also have the homothety (sending $f_C$ to $rf_C$ for $r\neq 0$) as well.

We will call such an equation $f_C$ \emph{allowable} if $a_{20}\neq 0$,
$a_{04}\neq 0$, and $f_C=0$ has only ade singularities otherwise.
This proves the following.

\begin{prop}\label{normalA12}
Define the action of $\bbG_m^3$ on $\bbA^{12}$
(with coordinates $\{a_{ij}\;|\; ij \in Q\}$)
by $(r,s,t)\cdot a_{ij} = rs^it^j a_{ij}$.
The isomorphism classes of possible branch curves $C$
are in $1$-$1$ correspondence with the allowable $\bbG_m^3$-orbits of this action.
\end{prop}

We have yet one more normalization to propose.
Since by Lemma \ref{normal} the coefficients $a_{20}$ and $a_{04}$ are nonzero,
we may use elements of $\bbG_m^3$ to make these two coefficients equal to one.
The subgroup of $\bbG_m^3$ that fixes these coefficients
is equal to the set of elements of the form $(r,s,t)=(t^{-4},\pm t^2,t)$;
these form a subgroup isomorphic to $\bbG_m \times \bbZ/2\bbZ$.

Using this further normalization, we have the following.
Define $\bar{Q} = Q\setminus\{20,04\} = \{12,13,14,22,23,24,31,32,33,34\}$
to be the set of remaining indices.

\begin{prop}\label{normalA10}
Define the action of $\bbG_m\times\bbZ/2\bbZ$ on $\bbA^{10}$
(with coordinates $\{a_{ij}\;|\; ij \in \bar{Q}\}$)
by $(t\in\bbG_m,\epsilon\in\{\pm 1\})\cdot a_{ij} = t^{2i+j-4}\epsilon^i a_{ij}$.
The isomorphism classes of possible branch curves $C$
are in $1$-$1$ correspondence with the allowable $\bbG_m\times\bbZ/2\bbZ$-orbits of this action.
\end{prop}

\section{The $\bbG_m\times\bbZ/2\bbZ$ quotient of $\bbA^{10}$ }
\label{z2}

In order to construct the quotient $\mathbb{A}^{10}/(\bbG_m\times\bbZ/2\bbZ)$
we can first take the quotient by the $\bbZ/2\bbZ$ action;
for this we identify the invariants for that action.
Since the element $(1,-1,1)$ acts on $a_{ij}$ trivially if $i$ is even,
and as multiplication by $-1$ if $i$ is odd, 
it follows that a Laurent monomial $\prod_{ij} a_{ij}^{m_{ij}}$ is invariant if and only if
$m_{12}+m_{13}+m_{14}+m_{31}+m_{32}+m_{33}+m_{34}$ is even.
Therefore, we have the following generators for the invariant Laurent monomials:

\begin{align*}
w_{12} &=a_{12}^2\\
w_{ij} &= a_{ij}/a_{12} \quad ij\in\{13,14,31,32,33,34\}\\
w_{ij} &= a_{ij} \quad ij \in  \{22,23,24\} 
\end{align*}

The only condition on a Laurent monomial in the $w_{ij}$ variables is that,
as monomial in the $a_{ij}$ variables, it has all nonnegative exponents.
Hence we see that the quotient $\mathbb{A}^{10}/(\bbZ/2\bbZ)$
is isomorphic to the affine toric variety $Y=\Spec\,k[N]$,
where $N$ is the monoid of monomials $\prod_{ij} w_{ij}^{n_{ij}}$
whose exponents $\{n_{ij}\}$ satisfy the following ten inequalities

\begin{align}
2n_{12}-n_{13}-n_{14}-n_{31}-n_{32}-n_{33}-n_{34} \geq 0\\
n_{ij}\geq 0 \quad ij\neq 12
\end{align}

That is, 
\[
k[N]\simeq
k[w_{22},w_{23},w_{24}][w_{12},w_{12}w_{ij},w_{12}w_{ij}w_{kl}]
= k[a_{22},a_{23},a_{24},][a_{ij}a_{kl}]
\]
 for $ij,kl\in \{12,13,14,31,32,33,34\}$.

Now, recall the elements $(t^{-4},t^2,t)$ act on $a_{ij}$ as multiplication by $t^{2i+j-4}$;
this same exponent applies to the induced action on $w_{ij}$.
Thus, regrading the ring $k[N]$ so that $w_{ij}$ has degree $2i+j-4$,
we have that
\[
\mathbb{A}^{10}/(\mathbb{G}_m \times \mathbb{Z}/2\mathbb{Z}) 
\simeq (\mathbb{A}^{10}/( \mathbb{Z}/2\mathbb{Z}))/\mathbb{G}_m 
\simeq Y/\mathbb{G}_m 
\simeq \Proj\,k[N]
\]

In particular, the basic open sets
of $\mathbb{A}^{10}/(\mathbb{G}_m \times \mathbb{Z}/2\mathbb{Z})$
are determined by inverting the generators of $k[N]$ of strictly positive degree,
and then taking $\Spec$ of the degree zero part of the corresponding fraction ring.

Given a generator $f=\prod_{ij\in \bar{Q}}w_{ij}^{n_{ij}}\in k[N]$,
a Laurent monomial $\prod_{ij\in \bar{Q}}w_{ij}^{p_{ij}}$ 
lies in the degree zero part of the fraction ring $k[N][1/f]$ 
if and only if the exponents $\{p_{ij}\}$ satisfy
\[
\sum_{ij \in \bar{Q}} (2i+j-4)p_{ij} = 0
\]
(which allows us to solve for $p_{13}$ in terms of the other exponents);
in addition we must have:
\begin{equation}\label{12vector}
2p_{12}+p_{14}+2p_{22}+3p_{23}+4p_{24}+2p_{31}+3p_{32}+4p_{33}+5p_{34}\geq 0
\end{equation}
whenever $2n_{12}=n_{13}+n_{14}+n_{31}+n_{32}+n_{33}+n_{34}$;
\begin{equation}\label{13vector}
2p_{14}+2p_{22}+3p_{23}+4p_{24}+3p_{31}+4p_{32}+5p_{33}+6p_{34}\leq 0
\end{equation}
whenever $n_{13}=0$; and
\begin{equation}\label{ijvectors}
p_{ij}\geq 0
\end{equation}
for $ij\neq 12,13$ whenever $n_{ij}=0$.

Therefore, if one fixes the monomial $f$ to be inverted,
we obtain conditions on the $p_{ij}$'s that yield an affine toric variety;
these toric varieties can then be glued together to give the toric description of $\Proj\, k[N]$.
For this, it is enough to invert the generators of $N$ of positive degree.

Specifically, we work in the $\bbZ^9 \subset \bbR^9$
with coordinates indexed by $\bar{Q}\setminus\{13\} = \{12,14,22,23,24,31,32,33,34\}$.
Given a positive degree monomial generator $f$,
the relevant inequalities (\ref{12vector}), (\ref{13vector}), (\ref{ijvectors}) that apply
define a cone (as the intersection of the half-spaces given by these conditions)
that determine the monoid and thus the coordinate ring of that open subset of the $\Proj$.
The coefficients of these inequalities
(in terms of the $p_{ij}$ coordinates)
give vectors in the dual space that define the dual cones as their convex hull;
this viewpoint realizes the $\Proj$
as being defined by the familiar fan of dual cones.

The following table gives the relevant cones for each positive degree generator of $N$.

\[
\begin{matrix}
\text{generator of }k[N] & \text{inequalities that define the cone} \\
w_{2j}, j=2,3,4 & (\ref{12vector}), (\ref{13vector}), p_{k\ell}\geq 0 \text{ for } k\ell \neq 12,2j \\
w_{12}w_{13} \text{ or } w_{12}^2w_{13} & p_{k\ell}\geq 0 \text{ for } k\ell \neq 12 \\
w_{12}w_{ij} \text{ or } w_{12}^2w_{ij} , ij \in \{14,31,32,33,34\} & (\ref{13vector}), p_{k\ell}\geq 0 \text{ for } k\ell \neq 12, ij \\
w_{12}w_{13}^2 & (\ref{12vector}), p_{k\ell}\geq 0 \text{ for } k\ell \neq 12 \\
w_{12}w_{13}w_{ij}, ij \in \{14,31,32,33,34\} & (\ref{12vector}), p_{k\ell}\geq 0 \text{ for } k\ell \neq 12, ij \\
w_{12}w_{ij}w_{i'j'}, ij, i'j' \in \{14,31,32,33,34\} & (\ref{12vector}), (\ref{13vector}),  p_{k\ell}\geq 0 \text{ for } k\ell \neq 12, ij,i'j' \\
\end{matrix}
\]

The reader will note that some of these cones are the same; that is not a complete surprise.
In addition, some cones are subsets of others, which correspond to larger open subsets of the $\Proj$.  Hence for the $\Proj$ construction we can focus on the minimal cones only.
There are nine of these: the three in the first row, the one in the fourth row, and the five in the sixth row (when $i'j'=ij$) in the above table, and are described by the following:
\begin{align*}
C_{13}: &(\ref{12vector}) \text{ and } p_{k\ell} \geq 0 \text{ for } k\ell \neq 12 \\
C_{ij} (ij \neq 12,13): &(\ref{12vector}), (\ref{13vector}), \text{ and } p_{k\ell} \geq 0 \text{ for } k\ell \neq 12, ij
\end{align*}
(The rationale for the indexing will become clear in the next section.)

This $\Proj$ is not a complete variety however;
the monomial $w_{12}$ has degree zero and appears in every toric coordinate ring.
It gives an affine invariant of the marked RES of index two $Y$, which determines the $J$-invariant of the (reduced) double fiber and has the following geometric interpretation. 

Considering $Y$ as a double cover of $\bbP^1\times \bbP^1$ as in Section \ref{double},
we see that, on the second blowup $\bbF$,
then we have an affine coordinate $\eta$ on the second exceptional curve $e_2$
such that the four branch points on $e_2$ are at $0$, $\infty$, and the two roots of the quadratic equation
$$
\eta^2+w_{12}\eta+1=0
\label{doublept}
$$
This implies that the genus one curve $F_0$ that becomes the reduced double fiber
has $J$-invariant equal to $\frac{4}{3}(3-w_{12}^2)/({4-w_{12}^2})$.

The discriminant of the quadratic is seen here as giving the criterion for the double fiber to be singular: when $w_{12}=\pm 2$, the double fiber must be of type $I_n$ 
for some $1 \leq n\leq 9$ \cite[Proposition 5.1.8]{dc}.

In any case the construction, along with Proposition \ref{normalA10}, 
gives us the following;
it is essentially the first part of Theorem \ref{thm:main}.

\begin{thm}
The subset of $\Proj(k[N])$ corresponding to orbits of allowable branch curves is a moduli space for marked RESs of index two.
\end{thm}

\section{The $\bbG_m^3$ quotient of $\bbA^{12}$}
\label{sec:A12quotient}

For this analysis we generally follow the prescriptions of Chapter 12 of \cite{invariant}
which we find convenient;
the reader may also consult Chapter 14 of \cite{cox}.

The GIT quotient of the (diagonalized) action of $\bbG_m^3$ on $\bbA^{12}$ 
as described in Proposition \ref{normalA12} 
is a toric variety which depends on the choice of a $\bbG_m^3$-linearization $L_\chi$ 
of the trivial line bundle, 
which is determined by a character $\chi$ of $\bbG_m^3$.
Such a character may be written as $(r,s,t)\mapsto r^{\alpha_1}s^{\alpha_2}t^{\alpha_3}$,
and depends on the three integers $\alpha_i$.

Briefly, the construction is as follows.
Given a monomial $f:= \prod_{ij \in Q} a_{ij}^{m_{ij}}$,
the action of an element $(r,s,t) \in \bbG_m^3$ is of the form
$
(r,s,t)\cdot f = r^{\beta_1}s^{\beta_2}t^{\beta_3} f
$
where the $\beta_i$'s are determined by
\begin{equation}
A\cdot \underline{m} = \begin{pmatrix}
1& 1& 1& 1& 1& 1& 1& 1& 1& 1& 1& 1\\
2& 0& 1& 1& 1& 2& 2& 2& 3& 3& 3& 3 \\
0& 4& 2& 3& 4& 2& 3& 4& 1& 2& 3& 4
\end{pmatrix}
\cdot \underline{m} =
\begin{pmatrix}
\beta_1 \\
\beta_2\\
\beta_3
\end{pmatrix}
\label{system}
\end{equation}
where $\underline{m} = (m_{20},m_{04},m_{12},m_{13},m_{14},m_{22},m_{23},m_{24},m_{31},m_{32},m_{33},m_{34})^t$.

Such a monomial is then $\chi$-homogeneous of degree $d$
if $\beta_i = d\alpha_i$ for each $i$,
which gives the equation
\begin{equation}\label{Asystem}
A \cdot \underline{m} = d\underline{\alpha};
\end{equation}
these monomials form the ring of invariant sections of the linearized line bundle $L_\chi$
which determines the toric GIT quotient.

In particular, if we define $S_d$ to be the space of monomials
$\underline{a}^{\underline{m}}=\prod_{ij \in Q} a_{ij}^{m_{ij}}$,
where $\underline{m}$ is a solution to (\ref{Asystem}) with $m_{ij}$ and $d$ non-negative,
we obtain an isomorphism of finitely generated $k-$algebras
\[
k[S]\doteq \bigoplus_{d\geq 0} k[S_d] \simeq \bigoplus_{d\geq 0} \Gamma(\mathbb{A}^{12}, \mathcal{L}_{\chi}^{\otimes d})^{\bbG_m^3}
\]
and the GIT quotient $\bbA^{12}//_\chi \bbG_m^3$ determined by the character $\chi$
is the $\Proj$ of this graded ring.

Note that the rows of $A$ determine a particular embedding $\bbG_m^3 \hookrightarrow \bbG_m^{12}$, hence $A$ induces a map between the corresponding character groups, and we have the following short exact sequence
\begin{equation}
\xymatrix{0 \ar[r] & M \ar[r]^{B^{t}} & \mathbb{Z}^{12} \ar[r]^A & \mathbb{Z}^3 \ar[r] & 0}
\label{ses}
\end{equation}
where $B$ is the matrix whose rows give us a basis for the kernel lattice of $A$, which we denote by $M$, and $B^t$ denotes the transpose of $B$. 

Concretely, we may take as $B$ the matrix 
\begin{equation}
\begin{pmatrix}
B_{20} & B_{04} & B_{12} & B_{13} & B_{14} & B_{22} & B_{23} & B_{24} & B_{31} & B_{32} & B_{33} & B_{34} \\
-1& -1& 2& 0& 0& 0& 0& 0& 0& 0& 0& 0 \\ 
0& 0& 1& -2& 1& 0& 0& 0& 0&  0& 0& 0 \\
-1& 0& 2& -2& 0& 1& 0& 0& 0& 0& 0& 0 \\
-1& 0& 3& -3& 0& 0& 1& 0& 0& 0& 0& 0\\
-1& 0& 4& -4& 0& 0& 0& 1& 0& 0& 0& 0\\
-1& 1& 2& -3& 0& 0& 0& 0& 1& 0& 0& 0\\
-1& 1& 3& -4& 0& 0& 0& 0& 0& 1& 0& 0\\
-1& 1& 4& -5& 0& 0& 0& 0& 0& 0& 1& 0\\
-1& 1& 5& -6& 0& 0& 0& 0& 0& 0& 0& 1
\end{pmatrix}
\label{matrixB}
\end{equation}

The twelve columns of $B$ will determine the cones of the fan that exhibits the GIT quotient $\mathbb{A}^{12}//_{\chi} \bbG_m^3$ as a toric variety.  We note that the final ten columns correspond exactly to the conditions (\ref{12vector}), (\ref{13vector}), (\ref{ijvectors}) 
that were found in our toric description of the $\Proj$ defining the quotient
of $\bbA^{10}$ by the $\bbG_m\times\bbZ/2\bbZ$ group.

\subsection{Support sets and the semistable locus}

Given a monomial $\underline{a}^{\underline{m}}$ which lies in $S_d$ for some $d$, 
we define $D(\underline{m})$ to be the invariant open set
consisting of all the points of $\bbA^{12}$ 
where $\underline{a}^{\underline{m}}$ does not vanish. 
Thus, by definition, we have
\[
(\bbA^{12})_{\chi}^{ss} = \bigcup_{d\geq 0} \bigcup_{\underline{a}^{\underline{m}} \in S_d} D(\underline{m})
\]

Since $D(\underline{m})$ only depends on the set of indices $ij$ for which $m_{ij}\neq 0$, 
it is convenient to define the support of a monomial $\underline{a}^{\underline{m}}$ 
to be the subset of the set of indices $Q$
where the variable $a_{ij}$ appears with a strictly positive exponent:
\[
\Supp(\underline{a}^{\underline{m}}) := \{ij\;|\;m_{ij}>0\}.
\]

To any subset $I \subset Q$ we can then associate an invariant open
$D(I)\subset \mathbb{A}^{12}$
consisting of points with non-zero coordinates indexed precisely by $I$.
And since $D(\underline{m})=D(I)$ for $I=\Supp(\underline{a}^{\underline{m}})$,
and $D(I) \subset D(J)$ if $J \subset I$,
it follows that
\[
(\mathbb{A}^{12})_{\chi}^{ss} = \bigcup_{\text{minimal}\,I\in\Supp(S)} D(I)
\]
where $\Supp(S)$ denotes the poset formed by the set of all support sets of all invariant monomials in $k[S]$.

Now, if $I$ is a support set for an invariant monomial in $S_d$ for some $d$, 
then $D(I)$ is an affine variety with coordinate ring isomorphic to $k[Z_{I}]$, 
where $Z_{I}$ is the monoid of monomials $\underline{a}^{\underline{m}}$ 
such that $m_{ij}\geq 0$ precisely when $ij \notin I$. 

Since $D(I)$ is invariant under the action of $\bbG_m^3$, 
we can consider the corresponding quotient $D(I)/\bbG_m^3$, 
whose coordinate ring is the subring of $k[Z_{I}]$ generated by invariant monomials. 
A monomial $\underline{a}^{\underline{m}}\in k[Z_I]$ is invariant if and only if 
$A\cdot \underline{m}=0$, i.e. $\underline{m}$ lies in the kernel lattice $M$. 
Thus, letting $M_{I}=\{\underline{a}^{\underline{m}}\in k[Z_I]\,;\, \underline{m} \in M\}$ 
we see that the coordinate ring of $D(I)/\bbG_m^3$ is $k[M_{I}]$, 
which we can easily recognize as the coordinate ring of an affine toric variety.

Concretely, let $\{e_{ij}\}$ denote the standard basis of $\mathbb{Z}^{12}$ 
(indexed by $Q$) 
and let $\{e_{ij}^*\}$ denote the dual basis. 
A monomial $\underline{a}^{\underline{m}}$ lies in $k[M_{I}]$ if and only if 
$m_{ij}=e_{ij}^*(\underline{m})\geq 0$ whenever $ij\notin I$. 
That is, $M_{I}$ is determined by an intersection of half-spaces -- a cone. 
The dual cone, which we denote by $\sigma_I$, 
is the cone generated by the columns of $B$ which are indexed by $Q \setminus I$.

The set of all cones $\{\sigma_I\}$ 
where $I$ runs over all the minimal support sets in $\Supp(S)$ 
form a fan $\Sigma_{\chi}$ in $M^* \otimes \mathbb{R}\simeq \mathbb{R}^9$ 
and the quotient $(\mathbb{A}^{12})^{ss}//_\chi \bbG_m^3$ 
is the toric variety associated to this fan. 

When $\Sigma_{\chi}$ is simplicial there are no strictly semi-stable points \cite[Proposition 12.1]{invariant} and $(\mathbb{A}^{12})^{ss}//_\chi \bbG_m^3=(\mathbb{A}^{12})^{s}/_\chi \bbG_m^3$ is a geometric quotient.

\subsection{Support sets and the secondary fan} 

As the character $\chi$ varies, 
the GIT quotient $\bbA^{12}//_\chi \bbG_m^3$ changes, 
but there are only finitely many distinct quotients up to isomorphism. 
These are parametrized by finitely many disjoint chambers lying in a fan whose support is the convex cone generated by all the columns of $A$. This fan, called the secondary fan, lives inside Hom$(\bbG_m^3,\bbG_m)\otimes \mathbb{R}\simeq \mathbb{R}^3$, 
and on each chamber the quotient $\bbA^{12}//_\chi \bbG_m^3$ does not change. 
Since each column of $A$ has first coordinate equal to one, 
this chamber decomposition can be represented by Figure \ref{chambers}:

\begin{figure}[h]
\begin{tikzpicture}[line cap=round,line join=round,>=triangle 45,x=1cm,y=1cm,scale=2]
\clip(-0.1,-0.1) rectangle (3.1,4.1);
\draw [line width=0.5pt] (1,2)-- (2,0);
\draw [line width=0.5pt] (2,0)-- (3,1);
\draw [line width=0.5pt] (3,1)-- (1,2);
\draw [line width=0.5pt] (1,3)-- (2,0);
\draw [line width=0.5pt] (1,2)-- (1.4,1.8);
\draw [line width=0.5pt] (1.4,1.8)-- (2,0);
\draw [line width=0.5pt] (2,0)-- (1,2);
\draw [line width=0.5pt] (0,4)-- (1,2);
\draw [line width=0.5pt] (0,4)-- (1,3);
\draw [line width=0.5pt] (0,4)-- (3,4);
\draw [line width=0.5pt] (3,4)-- (3,1);
\draw [line width=0.5pt] (3,4)-- (2,2);
\draw [line width=0.5pt] (3,4)-- (1,2);
\draw [line width=0.5pt] (2,4)-- (1,3);
\draw [line width=0.5pt] (1,4)-- (1,2);
\draw [line width=0.5pt] (1,4)-- (3,2);
\draw [line width=0.5pt] (1,4)-- (3,3);
\draw [line width=0.5pt] (2,4)-- (2,0);
\draw [line width=0.5pt] (1,2)-- (3,2);
\draw [line width=0.5pt] (3,3)-- (2,0);
\draw [line width=0.5pt] (3,3)-- (2,2);
\draw [line width=0.5pt] (3,3)-- (1,2);
\draw [line width=0.5pt] (3,3)-- (1,3);
\draw [line width=0.5pt] (3,3)-- (0,4);
\draw [line width=0.5pt] (1,4)-- (2,2);
\draw [line width=0.5pt] (1,4)-- (2,0);
\draw [line width=0.5pt] (1,4)-- (3,1);
\draw [line width=0.5pt] (2,2)-- (3,1);
\draw [line width=0.5pt] (1,3)-- (2,2);
\draw [line width=0.5pt] (1,3)-- (3,4);
\draw [line width=0.5pt] (2,4)-- (3,3);
\draw [line width=0.5pt] (2,4)-- (1,2);
\draw [line width=0.5pt] (2,4)-- (3,2);
\draw [line width=0.5pt] (2,4)-- (3,1);
\draw [line width=0.5pt] (1,3)-- (3,2);
\draw [line width=0.5pt] (2,3)-- (3,1);
\draw [line width=0.5pt] (3,2)-- (2,0);
\draw [line width=0.5pt] (2,3)-- (0,4);
\draw [line width=0.5pt] (0,4)-- (3,2);
\draw [line width=0.5pt] (2,0)-- (1.5714285714285714,1.7142857142857142);
\draw [line width=0.5pt] (1.5714285714285714,1.7142857142857142)-- (1.4,1.8);
\draw [line width=0.5pt] (1.4,1.8)-- (2,0);
\draw [line width=0.5pt] (3,4)-- (2,0);
\draw [line width=0.5pt] (2.6,1.2)-- (2,0);
\draw [line width=0.5pt] (2,0)-- (3,1);
\draw [line width=0.5pt] (3,1)-- (2.6,1.2);
\draw [line width=0.5pt] (0,4)-- (1,2);
\draw [line width=0.5pt] (0,4)-- (1,3);
\draw [line width=0.5pt] (1,2)-- (1,3);
\begin{scriptsize}
\draw [fill=blue] (3,4) circle (1pt);
\draw [fill=blue] (3,3) circle (1pt);
\draw [fill=blue] (3,2) circle (1pt);
\draw [fill=blue] (3,1) circle (1pt);
\draw [fill=blue] (2,4) circle (1pt);
\draw [fill=blue] (2,3) circle (1pt);
\draw [fill=blue] (2,2) circle (1pt);
\draw [fill=blue] (2,0) circle (1pt);
\draw [fill=blue] (1,4) circle (1pt);
\draw [fill=blue] (1,3) circle (1pt);
\draw [fill=blue] (1,2) circle (1pt);
\draw [fill=blue] (0,4) circle (1pt);
\end{scriptsize}
\end{tikzpicture}
\caption{Chambers}
\label{chambers}
\end{figure}
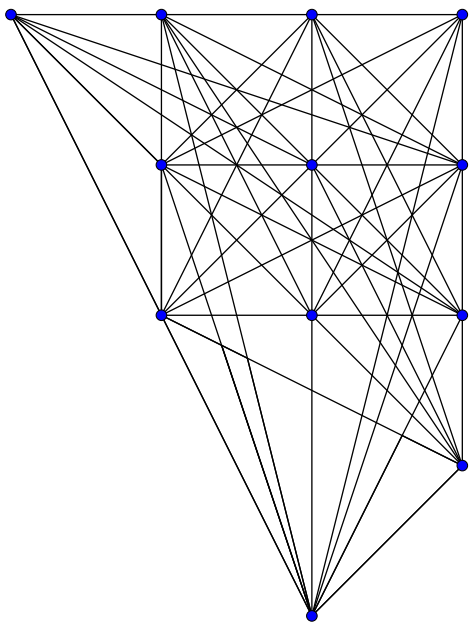

Given any character $\chi$ we can then identify the corresponding support sets using Lemma \ref{support} below.
For a finite set of points $P = \{g_i\} \subset \bbR^2$,
we define the \emph{strict convex hull} of the set $P$
to be the set of points of the form $\sum_i r_i g_i$ where $\sum_i r_i =1$
and each $r_i$ is strictly positive.

\begin{lm}
Given a character $\chi: (r,s,t) \mapsto r^{\alpha_1}s^{\alpha_2}t^{\alpha_3}$,
with $\alpha_1>0$, let $b=\alpha_2/\alpha_1$ and $c=\alpha_3/\alpha_1$.
A subset $I \subset Q$ is a support set for an invariant monomial in $S_d$ for some $d$ 
if and only if the point $(b,c)$ lies in the strict convex hull of $\{(i,j)\,;\,ij \in I\}$.
\label{support}
\end{lm}

\begin{proof}
We note that all numbers involved are rational.

First suppose $I = \Supp(\prod_{ij \in I} a_{ij}^{m_{ij}})$, 
where each $m_{ij}$ appearing here is at least one,
and the monomial is invariant.
The invariance implies that we have the equation (\ref{Asystem}), 
which we may then write as
\[
\sum_{ij \in I} m_{ij}\begin{pmatrix} 1 \\ i \\ j \end{pmatrix} =
d\begin{pmatrix} \alpha_1 \\ \alpha_2 \\ \alpha_3 \end{pmatrix}.
\]
Dividing all quantities by $d\alpha_1$ gives
\[
\sum_{ij \in I} \frac{m_{ij}}{d\alpha_1} \begin{pmatrix} 1 \\ i \\ j \end{pmatrix} =
\begin{pmatrix} 1 \\ b \\ c \end{pmatrix}.
\]
If we define $r_{ij} = m_{ij} /d \alpha_1$, 
we have $\sum_{ij \in I} r_{ij} = 1$, and
$\sum_{ij \in I} r_{ij} \begin{pmatrix} i \\ j \end{pmatrix} = \begin{pmatrix} b \\ c \end{pmatrix}$,
which shows that the point $(b,c) \in \bbR^2$ is in the strict convex hull of the points $(i,j)$ in $I$ as required.

Conversely, if $(b,c)$ is in the strict convex hull,
then by definition there are strictly positive rational numbers $r_{ij}$ such that
\[
\sum_{ij \in I} r_{ij} \begin{pmatrix} 1 \\ i \\ j \end{pmatrix} =
\begin{pmatrix} 1 \\ b \\ c \end{pmatrix}
\]
and multiplying through by $\alpha_1$ gives
\[
\sum_{ij \in I} \alpha_1 r_{ij} \begin{pmatrix} 1 \\ i \\ j \end{pmatrix} =
\begin{pmatrix} \alpha_1 \\ \alpha_2 \\ \alpha_3 \end{pmatrix}
\]
Now choose $d$ so that $m_{ij} = d \alpha_1 r_{ij} \in \bbZ$ for all $ij \in I$.
Multiplying through by $d$ gives the equation (\ref{Asystem})
and so the monomial $\prod_{ij \in I} a_{ij}^{m_{ij}}$ is in $S_d$, 
and it is invariant.
The strictness gives us that $r_{ij}$, and hence $m_{ij}$, 
is strictly positive for each $ij \in I$,
and so $I$ is indeed the support set of this monomial.
\end{proof}

We also prove

\begin{lm}
If $I$ is a minimal support set for an invariant monomial in $S_d$ for some $d$, 
then $|I|\leq 3$.
\label{size}
\end{lm}

\begin{proof}
Suppose $|I| \geq 4$.
Let $P$ be the polygon which is the convex hull of $I$,
so that $(b,c)$ is in the interior of $P$.
Let $p_1,\ldots,p_k$ be the vertices of $P$, 
in clockwise order around the boundary of $P$.
Note that each $p_i$ is in $I$.
If $(b,c)$ lies on one of the interior line segments of $P$ joining $p_1$ to $p_i$ 
(for $3 \leq i \leq k-1$),
then $J = \{p_1,p_i\}$ contains $(b,c)$ in the interior of its strict convex hull
(the line segment),
and so $I$ is not minimal.

If $(b,c)$ does not lie on any of those interior line segments, 
then $(b,c)$ must be in the interior of one of the triangles with vertices 
$p_1, p_i, p_{i+1}$ for some $i$ with $2 \leq i \leq k-1$
(where we set $p_{k+1}=p_1$).
Hence those three vertices give a smaller set of indices than $I$ which are in $\Supp(S)$,
and again we conclude that $I$ is not minimal.
\end{proof}

In particular, 

\begin{prop}
\end{prop}
\begin{enumerate}[(i)]
\item $(\bbA^{12})^{ss}//_\chi \bbG_m^3=\emptyset$ 
for every character $\chi$ such that the point $(b,c)$ 
lies outside the quadrilateral in Figure \ref{chambers}. 
\item $(\bbA^{12})^{ss}//_\chi \bbG_m^3=(\bbA^{12})^{s}/_\chi \bbG_m^3$ 
for every character $\chi$ such that the point $(b,c)$ does not lie in one of the line segments in Figure \ref{chambers}. 
\item The semistable locus $(\bbA^{12})_{\chi}^{ss}$ always contains the open subset 
\[
X\doteq \{(a_{ij})_{ij \in Q}
\,;\,a_{04},a_{20},a_{31},a_{34}\neq 0\}
\]
for any choice of character $\chi$ corresponding to a point $(b,c)$ 
that lies in the interior of the quadrilateral in Figure \ref{chambers}.
\end{enumerate}

\begin{proof}
\begin{enumerate}[(i)]
\item It follows from Lemma \ref{support}.
\item It follows from Lemma \ref{size} and \cite[Proposition 12.1]{invariant}
\item It suffices to observe that for any such character $\chi$, if $(a_{ij})_{ij \in Q}\in X$, then $(a_{ij})_{ij \in Q}$ cannot be unstable. In fact, we can always find a 
support set $I$ so that $(a_{ij})_{ij \in Q}\notin V(I)$. This is because either $\{04,20,31\}$ or $\{04,31,34\}$ must be a support set  for some invariant monomial, independent of the choice of $\chi$.  
\end{enumerate}
\end{proof}

\subsection{The quotients that compactify the $\bbG_m \times \bbZ/2\bbZ$ quotient}

Among the possible chambers in Figure \ref{chambers} 
we are interested in those whose corresponding quotients  
compactify the $\bbG_m \times \bbZ/2\bbZ$ quotient we described in Section \ref{z2}. These special chambers are highlighted in Figure \ref{colors} below.

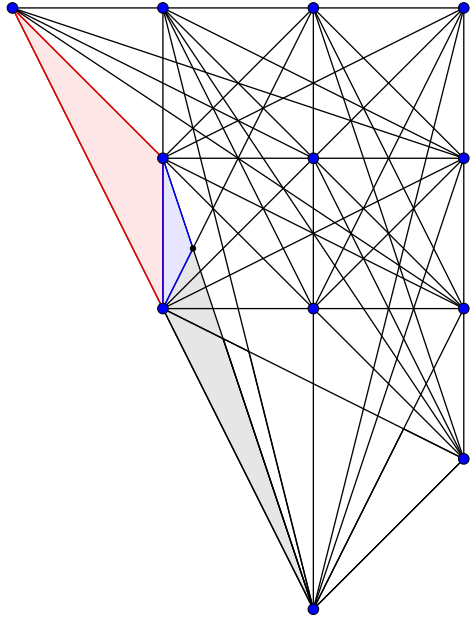
\begin{figure}[h]
\begin{tikzpicture}[line cap=round,line join=round,>=triangle 45,x=1cm,y=1cm,scale=2]
\clip(-0.1,-0.1) rectangle (3.1,4.1);
\fill[line width=0.5pt,,fill=black,fill opacity=0.1] (1,2) -- (1.2,2.4) -- (2,0) -- cycle;
\fill[line width=0.5pt,color=red,fill=red,fill opacity=0.1] (0,4) -- (1,2) -- (1,3) -- cycle;
\fill[line width=0.5pt,color=blue,fill=blue,fill opacity=0.1] (1.2,2.4) -- (1,2) -- (1,3) -- cycle;
\draw [line width=0.5pt] (1,2)-- (2,0);
\draw [line width=0.5pt] (2,0)-- (3,1);
\draw [line width=0.5pt] (3,1)-- (1,2);
\draw [line width=0.5pt] (1,3)-- (2,0);
\draw [line width=0.5pt,] (1,2)-- (1.4,1.8);
\draw [line width=0.5pt,] (1.4,1.8)-- (2,0);
\draw [line width=0.5pt,] (2,0)-- (1,2);
\draw [line width=0.5pt] (0,4)-- (1,2);
\draw [line width=0.5pt] (0,4)-- (1,3);
\draw [line width=0.5pt] (0,4)-- (3,4);
\draw [line width=0.5pt] (3,4)-- (3,1);
\draw [line width=0.5pt] (3,4)-- (2,2);
\draw [line width=0.5pt] (3,4)-- (1,2);
\draw [line width=0.5pt] (2,4)-- (1,3);
\draw [line width=0.5pt] (1,4)-- (1,2);
\draw [line width=0.5pt] (1,4)-- (3,2);
\draw [line width=0.5pt] (1,4)-- (3,3);
\draw [line width=0.5pt] (2,4)-- (2,0);
\draw [line width=0.5pt] (1,2)-- (3,2);
\draw [line width=0.5pt] (3,3)-- (2,0);
\draw [line width=0.5pt] (3,3)-- (2,2);
\draw [line width=0.5pt] (3,3)-- (1,2);
\draw [line width=0.5pt] (3,3)-- (1,3);
\draw [line width=0.5pt] (3,3)-- (0,4);
\draw [line width=0.5pt] (1,4)-- (2,2);
\draw [line width=0.5pt] (1,4)-- (2,0);
\draw [line width=0.5pt] (1,4)-- (3,1);
\draw [line width=0.5pt] (2,2)-- (3,1);
\draw [line width=0.5pt] (1,3)-- (2,2);
\draw [line width=0.5pt] (1,3)-- (3,4);
\draw [line width=0.5pt] (2,4)-- (3,3);
\draw [line width=0.5pt] (2,4)-- (1,2);
\draw [line width=0.5pt] (2,4)-- (3,2);
\draw [line width=0.5pt] (2,4)-- (3,1);
\draw [line width=0.5pt] (1,3)-- (3,2);
\draw [line width=0.5pt] (2,3)-- (3,1);
\draw [line width=0.5pt] (3,2)-- (2,0);
\draw [line width=0.5pt] (2,3)-- (0,4);
\draw [line width=0.5pt] (0,4)-- (3,2);
\draw [line width=0.5pt] (2,0)-- (1.5714285714285714,1.7142857142857142);
\draw [line width=0.5pt] (1.5714285714285714,1.7142857142857142)-- (1.4,1.8);
\draw [line width=0.5pt] (1.4,1.8)-- (2,0);
\draw [line width=0.5pt] (3,4)-- (2,0);
\draw [line width=0.5pt] (2.6,1.2)-- (2,0);
\draw [line width=0.5pt] (2,0)-- (3,1);
\draw [line width=0.5pt] (3,1)-- (2.6,1.2);
\draw [line width=0.5pt,color=red] (0,4)-- (1,2);
\draw [line width=0.5pt,color=red] (0,4)-- (1,3);
\draw [line width=0.5pt,color=red] (1,2)-- (1,3);
\draw [line width=0.5pt,color=blue] (1,2)-- (1,3);
\draw [line width=0.5pt,color=blue] (1.2,2.4)-- (1,3);
\draw [line width=0.5pt,color=blue] (1.2,2.4)-- (1,2);
\begin{scriptsize}
\draw [fill=blue] (3,4) circle (1pt);
\draw [fill=blue] (3,3) circle (1pt);
\draw [fill=blue] (3,2) circle (1pt);
\draw [fill=blue] (3,1) circle (1pt);
\draw [fill=blue] (2,4) circle (1pt);
\draw [fill=blue] (2,3) circle (1pt);
\draw [fill=blue] (2,2) circle (1pt);
\draw [fill=blue] (2,0) circle (1pt);
\draw [fill=blue] (1,4) circle (1pt);
\draw [fill=blue] (1,3) circle (1pt);
\draw [fill=blue] (1,2) circle (1pt);
\draw [fill=blue] (0,4) circle (1pt);
\draw [fill=black]	 (1.2,2.4) circle (0.5pt);
\end{scriptsize}
\end{tikzpicture}
\caption{Special chambers}
\label{colors}
\end{figure}

We describe next the GIT quotient $\bbA^{12}//_{\chi}\bbG_m^3$ for characters $\chi$ associated to points $(b,c)$ lying either in the red or blue chambers.

\subsubsection{The red chamber}

For any point $(b,c)$ lying in the shaded red chamber in Figure \ref{colors}, we explicitly show how the corresponding character $\chi$ determines a GIT quotient which compactifies the  $\bbG_m \times \bbZ/2\bbZ$ quotient.

First, fix such point $(b,c)$ and choose integers $\alpha_1,\alpha_2$ and $\alpha_3$ with $\alpha_1>0$ such that $b=\alpha_2/\alpha_1, c=\alpha_3/\alpha_1$ and $\chi$ is the character $(r,s,t)\mapsto r^{\alpha_1}s^{\alpha_2}t^{\alpha_3}$. By Lemma \ref{support}, the following is a complete list of minimal support sets for some invariant monomial:

\[
I_{ij}=\{04,20,ij\} \quad I'_{ij}=\{04,12,ij\}
\]
where $ij \neq 04,12,20$. Thus, 
\[
(\bbA^{12})_{\chi}^{ss} = \bigcup_{I\in\text{Supp}(\chi)} D(I)
\]
where we define $\Supp(\chi)=\{I_{ij},I'_{ij}\}$. We denote the resulting toric variety by $X_{\text{red}}$.

Now, indexing the rows of $B$ by $\{12,14,22,23,24,31,32,33,34\}$
we obtain that for $ij\neq 13$,
the cones $\check{\sigma}_{I_{ij}}$ 
are determined by the inequalities (\ref{12vector}) and (\ref{13vector}) from Section \ref{z2}, 
plus the seven inequalities $p_{kl}\geq 0$ for $kl\neq 12, ij$.
This exactly corresponds to the cone denoted by $C_{ij}$ in that section.
For $I_{13}$, we have that  $\check{\sigma}_{I_{13}}$
is determined by (\ref{12vector})
and the eight inequalities $p_{k\ell} \geq 0$, for $k\ell \neq 12$;
this is the cone which we called $C_{13}$ in the previous section.
Therefore this GIT quotient is indeed a compactification of the quotient obtained in the previous section.

The additional cones that fill out the (complete) fan are the ones defined by the $I'_{ij}$ support sets.

Similarly, when $ij\neq 13$ the cones $\check{\sigma}_{I'_{ij}}$ 
are determined by (\ref{13vector}) and the following set of inequalities:

\begin{align}
p_{12}+p_{22}+p_{23}+p_{24}+p_{31}+p_{32}+p_{33}+p_{34}\leq 0 \label{20vector}\\
p_{kl}\geq 0 \quad kl\neq 12, ij
\end{align}
while $\check{\sigma}_{I'_{13}}$ is determined only by (\ref{20vector}) and
$p_{k\ell} \geq 0$ for $k\ell \neq 12$.

\subsubsection{The blue chamber}

Consider now any point $(b,c)$ lying in the shaded blue chamber indicated in Figure \ref{colors}, and consider the associated character. Again, the corresponding quotient $\bbA^{12}//_{\chi}\bbG_m^3$ compactifies the quotient from Section \ref{z2}. 

Lemma \ref{support} tells us that, for this chamber, the following is a complete list of minimal support sets for some invariant monomial:

\[
I_{ij}=\{04,20,ij\}, \quad I'_{kl}=\{04,12,kl\}, \quad I''_{mn}=\{12,13,mn\}, \quad I'''_{mn}=\{12,14,mn\}
\]
where $ij \neq 04,12,20, kl \neq 04,12,13,14,20$ and $mn \neq 04,12,13,14$.

Note that the support sets $I_{ij}$ are also support sets for the red chamber, hence the same description applies here. That is, the cones $\check{\sigma}_{I_{ij}}$ are precisely the cones $C_{ij}$ from Section \ref{z2}.

In addition, the cones that fill out the (complete) fan are the ones defined by the remaining support sets $I'_{kl},I''_{mn}$ and $I'''_{mn}$. We denote the resulting toric variety by $X_{\text{blue}}$.

Note that the support sets $I'_{kl}$ are also support sets for the red chamber and we see that crossing from one chamber to the other, the cones $\check{\sigma}_{I'_{13}}$ and $\check{\sigma}_{I'_{14}}$ are replaced by the cones $\check{\sigma}_{I''_{mn}}$ and $\check{\sigma}_{I'''_{mn}}$, and vice-versa. This completely determines how the corresponding quotients change. 

\subsubsection{The wall-crossing} 

If we choose a character $\chi$ which corresponds to a point $(b,c)$ lying over the line segment joining $(1,2)$ and $(1,3)$, then Lemma \ref{support} tells us the following is a complete list of minimal support sets for some invariant monomial:
\[
I=\{12,13\} \qquad J=\{12,14\} \qquad I_{ij}=\{04,20,ij\} \qquad  I'_{kl}=\{04,12,kl\}
\]
where $ij \neq 04,12,20$ and $kl \neq 04,12,13,14,20$.

The cone $\sigma_I$ is a common refinement (see e.g. \cite[Chapter 3]{cox}) of the cone $\sigma_{I'_{13}}$ and the cones 
$\sigma_{I''_{mn}}$; while the cone $\sigma_J$ is a refinement of $\sigma_{I'_{14}}$ and $\sigma_{I'''_{mn}}$.

Since the cones $\sigma_I$ and $\sigma_J$ are non-simplicial the fan $\Sigma_{\chi}$ is non-simplicial and a flip describes how the quotient $\bbA^{12}//_{\chi}\bbG_m^3$  changes as we cross this ``wall".

Adopting the notations borrowed from \cite[Section 15.3]{cox}, we first observe the vectors $B_{ij}$ with $ij\notin I$ satisfy the relation
\begin{equation}
\sum b_{ij}B_{ij}=B_{20}-B_{04}+B_{22}+B_{23}+B_{24}+2B_{31}+2B_{32}+2B_{33}+2B_{34}=0
\label{rel}
\end{equation}

Now, if we let $J_{-}=\{ij\,;\,b_{ij}<0\}=\{04\}$, $J_0=\{ij\,;\,b_{ij}=0\}=\{14\}$ and $J_{+}=\{ij\,;\,b_{ij}>0\}=Q \backslash \{04,12,13,14\}$, we can define the following fans
\[
\Sigma_{\pm}=\{\sigma \,;\, \sigma \preceq \text{Cone}(B_{ij}), ij\neq 12,13, ij \notin J_{\mp}\}
\]

The fan $\Sigma_{+}$ has maximal cone $\sigma_{I'_{13}}$ and the maximal cones of $\Sigma_{-}$ are precisely the cones 
$\sigma_{I''_{mn}}$. In particular, we obtain a commutative diagram of surjective toric morphisms:

\begin{equation}
\xymatrix{
X_{\text{red}} \supset X_{\Sigma_{+}} \supset U_{\sigma_{I'_{13}}} \ar@{.>}[dr]_{\varphi_{+}} & & U_{\sigma_{I''_{mn}}} \subset X_{\Sigma_{-}}\subset X_{\text{blue}} \ar@{.>}[dl]^{\varphi_{-}} \\
 &U_{\sigma_{I}}&
}
\label{diagram}
\end{equation}
where $U_{\sigma}$ denotes the affine toric variety of the cone $\sigma$.

Each morphism $\varphi_{\pm}: \Sigma_{\pm}\to U_{\sigma_I}$ is birational with exceptional locus $V(\sigma_{J_{\pm}})$, where $\sigma_{J_{\pm}}=\text{Cone}(B_{ij}\,;\,ij\in J_{\pm})$ and $V(\sigma_{J_{\pm}})$ is the toric variety of Star$(\sigma_{J_{\pm}})$ (as defined in  \cite[(3.2.8)]{cox}).

Similarly, the vectors $B_{ij}$ with $ij\notin J$ also satisfy the relation
(\ref{rel}) and we can apply the same kind of analysis as before. 

For these indexes, $J_{-}=\{04\},J_{0}=\{13\}, J_{+}=Q\backslash\{04,12,13,14\}$ and
\[
\Sigma_{\pm}=\{\sigma \,;\, \sigma \preceq \text{Cone}(B_{ij}), ij\neq 12,14, ij \notin J_{\mp}\},
\]
so that the maximal fan of $\Sigma_{+}$ is now the cone $\sigma_{I'_{14}}$ and the maximal cones of $\Sigma_{-}$ are the cones $\sigma_{I'''_{mn}}$. Again, we obtain a commutative diagram analogous to (\ref{diagram}).

The analysis gives us the following statement, using Proposition \ref{normalA12}.

\begin{thm} Fix $(b,c)$ in the red or blue chambers.
Then the corresponding quotient $\bbA^{12}/\bbG_m^3$ contains
the quotient $\Proj(k[N])$ constructed in the previous section,
and compactifies that quotient.
In particular the subset of these quotients corresponding to orbits of allowable branch curves is a moduli space for marked RESs of index two.
\end{thm}

We remark that for $(b,c)$ in any of the shaded regions in Figure 2 this is still true.
We presented the analysis explicitly for the red and blue regions
(and the line segment separating them)
to illustrate the wall-crossing phenomenon in the toric GIT quotients for this situation.

A natural question then is: Given $(b,c)$ in the red or blue chambers, what are the curves being parametrized by the $\bbG_m^3$ quotient, but not by $\Proj(k[N])$? To answer this question we need to determine what are the points lying in $(\bbA^{12})^{ss}\backslash \bigcup D(I_{ij})$; or, equivalently,  we need to determine when is a point $(a_{ij})_{ij\in Q}$  satisfying one of the conditions below semi-stable.
\begin{enumerate}[(i)]
\item $a_{04}=0$
\item $a_{20}=0$
\item $a_{ij}=0\,,\,\forall\, ij \neq 04,12,20$
\end{enumerate} 

Ou analysis proves the following:

\begin{prop}
Fix $(b,c)$ in the red chamber. Then 
\[
(\bbA^{12})^{ss}\backslash \bigcup D(I_{ij})=\{(a_{ij})_{ij \in Q}
\,;
\, a_{20}=0\,\text{and}\,a_{12}\neq 0\}
\]
\end{prop}

And, similarly,

\begin{prop}
Fix $(b,c)$ in the blue chamber. Then $(\bbA^{12})^{ss}\backslash \bigcup D(I_{ij})$ consists of the set of points $(a_{ij})_{ij \in Q}$ satisfying one of the conditions below
\begin{enumerate}[(i)]
\item $a_{04}=a_{14}=0$ and $a_{12}\cdot a_{13} \cdot a_{mn}\neq 0$ for some $mn \neq 04,12,13,14$
\item $a_{04}=a_{13}=0$ and $a_{12}\cdot a_{14} \cdot a_{mn}\neq 0$ for some $mn \neq 04,12,13,14$
\item $a_{20}=a_{13}=a_{14}=0$ and $a_{04}\cdot a_{12} \cdot a_{kl}\neq 0$ for some $kl \neq 04,12,13,14,20$
\end{enumerate} 
\end{prop}

\section{Unstable loci}
\label{unstable}

By definition, a point $(a_{ij})_{ij \in Q}\in \bbA^{12}$ is unstable for the $\bbG_m^3$ action if and only if it is not semistable, if and only if every invariant monomial vanishes at it. In particular,
\[
(\mathbb{A}^{12})_{\chi}^{u} = \bigcap_{\text{minimal}\,I\in\Supp(S)} V(I)
\]
where $V(I)=\bbA^{12}\backslash D(I)$ consists of the set of common zeroes of the invariant monomials with support set $I$. That is, $V(I)=\left\{(a_{ij})_{ij \in Q}\,;\,\prod_{ij\in I} a_{ij}=0\right\}$.

This allows us to explicitly describe the unstable locus for any character $\chi$ associated to a point $(b,c)$ lying in the quadrilateral in Figure \ref{chambers}.   

For characters $\chi$ associated to points $(b,c)$ lying in the red chamber, the unstable locus is 
\[
\{ a_{04}=0\}\cup \{a_{12}=a_{20}=0\} \cup \{a_{ij}=0,\, \forall\, ij \neq 04,12,20\}
\]

Similarly, for characters associated to points lying in the blue chamber, one shows that the unstable locus is 
\begin{align*}
\{a_{04}=a_{12}=0\}\cup\{a_{12}=a_{20}=0\} \cup \{a_{04}=a_{13}=a_{14}=0\}\cup\\
  \{a_{04}=a_{mn}=0,\, \forall\, mn \neq 04,12,13,14\}\cup  
  \{a_{20}=a_{kl}=0,\, \forall\, kl \neq 04,12,13,14,20\}\cup \\
 \{a_{13}=a_{14}=a_{kl}=0,\, \forall\, kl \neq 04,12,13,14,20\}
\end{align*}

Note that, by Lemma \ref{normal}, when a point $(a_{ij})_{ij \in Q}\in \bbA^{12}$ describes a RES of index two, then it must be the case that $a_{04} \neq 0$ and $a_{20}\neq 0$. Hence for both chambers the only RES of index two that could a priori correspond to unstable points are those with branch curve $C$ given by 
\[
a_{04}v^3x^4+a_{12}uv^2x^2y^2+a_{20}u^2vy^4=0
\]

Now, such curve contains the curve $v=0$ as a component and the residual $(4,2)$ curve  has an additional tacnode at the point $p'$ given by $y=v=0$ (with a vertical tangent $v=0$). In particular, this $(4,2)$ curve cannot be irreducible by the genus formula. In fact it consists of two curves of bidegree $(2,1)$ which are tangent at the tacnode. After resolving the singularity at $p'$, the construction from Section \ref{double} yields a surface which is birational to a product of an elliptic curve and $\bbP^1$ and not a RES of index two.

\bibliography{ref}

\begin{bibdiv}
\begin{biblist}

\bib{alexeevK3}{misc}{
      author={Alexeev, V.},
      author={Brunyate, A.},
      author={Engel, P.},
       title={Compactifications of moduli of elliptic k3 surfaces: stable pair
  and toroidal},
        date={2021},
}

\bib{alexeevADE}{article}{
      author={Alexeev, V.},
      author={Thompson, A.},
       title={A{DE} surfaces and their moduli},
        date={2021},
     journal={J. Algebraic Geom.},
      volume={30},
      number={2},
       pages={331\ndash 405},
}

\bib{ascher}{misc}{
      author={Ascher, K.},
      author={Bejleri, D.},
       title={Stable pair compactifications of the moduli space of degree one
  del pezzo surfaces via elliptic fibrations},
        date={2018},
}

\bib{cox}{book}{
      author={Cox, D.~A.},
      author={Little, J.~B.},
      author={Schenck, H.~K.},
       title={Toric varieties},
      series={Graduate Studies in Mathematics},
   publisher={American Mathematical Society, Providence, RI},
        date={2011},
      volume={124},
}

\bib{invariant}{book}{
      author={Dolgachev, I.},
       title={Lectures on invariant theory},
      series={London Mathematical Society Lecture Note Series},
   publisher={Cambridge University Press, Cambridge},
        date={2003},
      volume={296},
}

\bib{dc}{book}{
      author={Dolgachev, I.},
      author={Cossec, F.},
       title={Enriques surfaces i},
      series={Progress in Mathematics},
   publisher={Birkh\"{a}user},
        date={1989},
      volume={76},
}

\bib{friedman}{article}{
      author={Friedman, R.},
       title={Vector bundles and {${\rm SO}(3)$}-invariants for elliptic
  surfaces},
        date={1995},
     journal={J. Amer. Math. Soc.},
      volume={8},
      number={1},
       pages={29\ndash 139},
}

\bib{hl}{incollection}{
      author={Heckman, G.},
      author={Looijenga, E.},
       title={The moduli space of rational elliptic surfaces},
        date={2002},
   booktitle={Algebraic geometry 2000, {A}zumino ({H}otaka)},
      series={Adv. Stud. Pure Math.},
      volume={36},
   publisher={Math. Soc. Japan, Tokyo},
       pages={185\ndash 248},
}

\bib{mirandaW}{article}{
      author={Miranda, R.},
       title={The moduli of {W}eierstrass fibrations over {${\bf P}^{1}$}},
        date={1981},
     journal={Math. Ann.},
      volume={255},
      number={3},
       pages={379\ndash 394},
}

\bib{az}{article}{
      author={Zanardini, A.},
       title={Stability of pencils of plane sextics and {H}alphen pencils of
  index two},
        date={2021},
      eprint={2101.03152},
}

\end{biblist}
\end{bibdiv}
\bibliographystyle{alpha}

\end{document}